\documentclass[12pt]{article}

\usepackage{amssymb,a4}
\usepackage{amsmath,amsfonts,amssymb,amsthm}

\newcommand{\1}{\mathbf {1}}
\newcommand{\Z}{{\mathbb Z}}
\newcommand{\C}{{\mathbb C}}

\newcommand{\wh}{{\widehat{\mathfrak h}}}

\newcommand{\mraff}{\mathrm{aff}}

\newtheorem{thm}{Theorem}[section]
\newtheorem{prop}[thm]{Proposition}
\newtheorem{lem}[thm]{Lemma}

\newtheorem{rmk}[thm]{Remark}
\newtheorem{definition}[thm]{Definition}

\begin{document}

\begin{center}
{\Large \bf  Representations of $\mathbb{Z}_{2}$-orbifold of the parafermion vertex operator algebra $K(sl_2,k)$}

\end{center}

\begin{center}
{ Cuipo Jiang$^{a}$\footnote{Supported by China NSF grants No.11771281 and No.11531004.}
and Qing Wang$^{b}$\footnote{Supported by
China NSF grants No.11622107 and No.11531004, Natural Science Foundation of Fujian Province
	No.2016J06002 and Fundamental Research Funds for the Central University No.20720160008.}\\
$\mbox{}^{a}$ School of Mathematical Sciences, Shanghai Jiao Tong University, Shanghai 200240, China\\
\vspace{.1cm}
$\mbox{}^{b}$ School of Mathematical Sciences, Xiamen University,
Xiamen 361005, China\\
}
\end{center}

\begin{abstract}
In this paper, the irreducible modules for the $\Z_{2}$-orbifold vertex operator subalgebra of the parafermion vertex operator algebra associated to the
irreducible highest weight modules for the affine Kac-Moody algebra
$A_1^{(1)}$ of level $k$ are classified and constructed.

\end{abstract}

\section{Introduction}
\def\theequation{1.\arabic{equation}}
\setcounter{equation}{0}

Coset construction and orbifold construction are two major ways to construct new vertex operator algebras from given ones. The parafermion vertex operator algebra $K(\mathfrak{g},k)$ is a special kind of coset construction. It is the commutant of a Heisenberg vertex operator subalgebra in the simple affine vertex operator algebra $L_{\hat{\mathfrak{g}}}(k,0)$, where $L_{\hat{\mathfrak{g}}}(k,0)$ is the integrable highest weight module with the positive integer level $k$ for affine Kac-Moody algebra $\hat{\mathfrak{g}}$ associated to a finite dimensional simple Lie algebra $\mathfrak{g}$. We denote $K(sl_{2},k)$ by $K_{0}$ and $L_{\hat{sl_{2}}}(k,0)$ by $L(k,0)$ in this paper. A systematic study of parafermion vertex operator algebras started from  \cite{DLY2} due to the connection of parafermion vertex operator algebras with $W$-algebras, Griess algebras and the moonshine vertex operator algebra. The structure and representation theory of the parafermion vertex operator algebras has been fully studied these years (see \cite{ALY1,ALY2,DLWY,DLY2,DW1,DW2,DW3,DR,JL1,JL2,Lam,Wang} etc.) Among the main results of the structure and representations for the parafermion vertex operator algebras, we see that the parafermion vertex operator algebra $K_0$ is no doubt the most important in the study of the general parafermion vertex operator algebra $K(\mathfrak{g},k)$. The reason comes from the structure of the parafermion vertex operator algebra $K(\mathfrak{g},k)$  studied in \cite{DW1}, that is,  $K_0$ is the building block of $K(\mathfrak{g},k)$, which shows that the role of  $K_0$ in the $K(\mathfrak{g},k)$ is similar to the role of $sl_2$ played in Kac-Moody algebras.

 On the other hand, parafermion vertex operator algebra  $K_0$ can be identified to certain $W$-algebra.
It was proved in \cite{DLY2} that the Virasoro vector $\omega$ and Virasoro primary vectors $W^{3}, W^{4}$ and $W^{5}$ of weight $3,4$ and $5$ are the strong generators of $K_{0}$. Thus the simple vertex operator algebra $K_{0}$ can be regarded as the quotient of $W$-algebra $W(2,3,4,5)$. Recently, a conjecture states that the parafermion vertex operator algebra $K_{0}$ is isomorphic to the $(k+1,k+2)$-minimal series $W$-algebra associated with $sl_k$ has been proved in \cite{ALY2}. Furthermore, both $C_2$-cofiniteness and rationality have also been confirmed in \cite{ALY1} and \cite{ALY2}. Specifically, the classification of irreducible modules for parafermion vertex operator algebra $K_{0}$ has been done in \cite{ALY1}.

 From \cite{DLY2}, we know that the full automorphism group of the parafermion vertex operator algebra $K_{0}$ for $k\geq 3$ is a group of order 2 generated by $\sigma$ which is defined  by $\sigma(h)=-h$, $\sigma(e)=f$, $\sigma(f)=e$, where $\{ h, e, f\}$ is a standard Chevalley basis of $sl_2$ with brackets $[h,e] = 2e$, $[h,f] = -2f$, $[e,f] = h$. Let $V$ be a vertex operator algebra and $G$  a finite automorphism group of $V$, the fixed-point subalgebra $V^{G}=\{v\in V| \ g.v=v,\ g\in G\}$ is called an orbifold vertex operator subalgebra. In this paper, we study the $\mathbb{Z}_{2}$-orbifold vertex operator subalgebra $K_{0}^{\sigma}$ of the parafermion vertex operator algebra $K_{0}$. More precisely, we classify and construct all the irreducible modules of the orbifold vertex operator subalgebra $K_{0}^{\sigma}$. From \cite{DM1}, we know that if $V$ is a vertex operator algebra with an automorphism $g$ of order $T$, and $M=\sum_{n\in\frac{1}{T}{\mathbb{Z}}_{+}}M(n)$ is an irreducible $g$-twisted admissible module of $V$, then
$M^{i}=\bigoplus_{n\in\frac{i}{T}+{\mathbb{Z}}_{+}}M(n)$ is an irreducible $V^{g}$-module for $i=0,\cdots, T-1$. From \cite{M2}, \cite{CM} (also see \cite{CKLR}), we know that the orbifold vertex operator subalgebra $K_{0}^{\sigma}$ is regular. The main difficulty to determine the irreducible modules of the orbifold vertex operator subalgebra $K_{0}^{\sigma}$ is to find the lowest weight vectors in each $M^{i}$ for our case $i=0,1, T=2$ and to find the lowest weight vectors of irreducible modules of the orbifold vertex operator subalgebra $K_{0}^{\sigma}$ from the irreducible modules of $K_0$.
Based on the classification results of the irreducible modules of $K_0$, we can deduce that if $k$ is odd, there are at most $\frac{k-1}{2}+1$ inequivalent irreducible $\sigma$-twisted modules of $K_0$, and if $k$ is even, there are at most $\frac{k}{2}+2$
inequivalent $\sigma$-twisted irreducible modules of $K_0$.
The problem becomes interesting when we construct explicitly the irreducible $\sigma$-twisted modules of $K_0$. In the case for $k$ being odd, we can follow the standard way to construct $\frac{k-1}{2}+1$ irreducible $\sigma$-twisted modules for $K_0$ from the $\sigma$-twisted modules of the affine vertex operator algebra $L(k,0)$ by analyzing the lowest weight in the grade zero of the admissible $\sigma$-twisted modules of the affine vertex operator algebra $L(k,0)$. However, the situation becomes different for the case of $k$ being even, since there are only $\frac{k}{2}+1$ different lowest weights in the grade zero of the admissible $\sigma$-twisted modules of the affine vertex operator algebra $L(k,0)$. So we need to find another lowest weight vector of $K_{0}$ in the $\sigma$-twisted modules of the affine vertex operator algebra $L(k,0)$. We achieve this by the highly nontrivial calculation in Section 3. It turns out that another lowest weight vector of $K_{0}$ is in the grade $\frac{1}{2}$ of the $\sigma$-twisted module constructed from the irreducible module $L(k,k/2)$ of the affine vertex operator algebra $L(k,0)$.  We achieve this with the help of the orbifold model $(V_{\mathbb{Z}\beta}^{+})^{\sigma}$ of the lattice vertex operator algebra $V_{\mathbb{Z}\beta}^{+}$ with $\langle \beta,\beta \rangle=6$, where $\sigma$ is an automorphism of $V_{\mathbb{Z}\beta}^{+}$ induced from $\sigma(\beta)=-\beta$. The reason is that from \cite{DLY2}, we know that if $k=4$, $K_0$ is isomorphic to the lattice vertex operator algebra $V_{\mathbb{Z}\beta}^{+}$ with $\langle \beta,\beta \rangle=6$, and we can write down all the lowest weights of the irreducible modules of the orbifold vertex operator subalgebra $(V_{\mathbb{Z}\beta}^{+})^{\sigma}$, which helps us to find the lowest weight vector in each $M^{i}$ for $i=0,1, T=2$ and the lowest weight vectors of irreducible modules of the orbifold vertex operator subalgebra $K_{0}^{\sigma}$ from the irreducible modules of $K_0$. As a result, the complete list of the lowest weight vectors with their weights for all the irreducible modules of $K_{0}^{\sigma}$ has been presented in Section 3. Finally we would like to point out that the clear list of the lowest weight vectors with their weights for all the irreducible modules of $K_{0}^{\sigma}$ is necessary for us to determine the quantum dimensions and fusion rules of the orbifold vertex operator subalgebra $K_{0}^{\sigma}$, which is our next work.

The paper is organized as follows. In Section 2, we recall some results about the parafermion vertex operator algebra $K_0$ and its orbifold vertex operator subalgebra $K_{0}^{\sigma}$. In Section 3, we first study the $\sigma$-twisted modules of parafermion vertex operator algebra $K_0$, and give the construction of all the irreducible $\sigma$-twisted modules of parafermion vertex operator algebra $K_0$, then we classify and construct the irreducible modules of the orbifold vertex operator subalgebra $K_{0}^{\sigma}$.

\section{Preliminary}
\label{Sect:V(k,0)}

In this section, we recall from \cite{DLY2}, \cite{DLWY} and \cite{ALY1} some basic results on the parafermion vertex
operator algebra associated to the
irreducible highest weight module for the affine Kac-Moody algebra
$A_1^{(1)}$ of level $k$ with $k$ being a positive integer. First we recall the notion of the parafermion vertex operator algebra.

We are working in the setting of \cite{DLY2}. Let $\{ h, e, f\}$
be a standard Chevalley basis of $sl_2$ with Lie brackets $[h,e] = 2e$, $[h,f] = -2f$,
$[e,f] = h$. Let $\widehat{sl}_2 = sl_2 \otimes \C[t,t^{-1}]
\oplus \C C$ be the affine Lie algebra associated to $sl_2$. Let $k \ge 1$
be an integer and
\begin{equation*}
V(k,0) = V_{\widehat{sl}_2}(k,0) = \mbox{Ind}_{sl_2 \otimes \C[t]\oplus \C
C}^{\widehat{sl}_2}\C
\end{equation*}
be the induced $\widehat{sl}_2$-module such that $sl_2 \otimes \C[t]$ acts
as $0$ and $C$ acts as $k$ on $\mathbf{1}=1$. Then $V(k,0)$ is a
vertex operator algebra generated by $a(-1)\1$ for $a\in sl_2$ such
that
$$Y(a(-1)\1,z) = a(z)=\sum_{n \in \Z} a(n)z^{-n-1}$$
where $a(n)=a\otimes t^n$, with the
vacuum vector $\1$ and the Virasoro vector
\begin{align*}
\omega_{\mraff} &= \frac{1}{2(k+2)} \Big( \frac{1}{2}h(-1)^2\1 +
e(-1)f(-1)\1 + f(-1)e(-1)\1 \Big)\\
&= \frac{1}{2(k+2)} \Big( -h(-2)\1 + \frac{1}{2}h(-1)^2\1 + 2e(-1)f(-1)\1
\Big)
\end{align*}
of central charge $\frac{3k}{k+2}$ (e.g. \cite{FZ},
\cite{Kac}, \cite[Section 6.2]{LL}).

Let $M(k)$ be the vertex operator subalgebra of $V(k,0)$
generated by $h(-1)\1$ with the Virasoro
element
$$\omega_{\gamma} = \frac{1}{4k}
h(-1)^{2}\1$$
of central charge $1$.


The vertex operator algebra $V(k,0)$ has a unique maximal ideal $\mathcal{J}$, which is generated by a weight $k+1$ vector $e(-1)^{k+1}\1$ \cite{Kac}. The quotient algebra $L(k,0)=V(k,0)/\mathcal{J}$ is a
simple, rational  vertex operator algebra as $k$ is a positive
integer (cf. \cite{FZ}, \cite{LL}). Moreover, the image of $M(k)$
in $L(k,0)$ is isomorphic to $M(k)$ and will be
denoted by $M(k)$ again. Set
\begin{equation*}
 K(sl_2,k)=\{v \in {\mathcal{L}}(k,0)\,|\, h(m)v =0
\text{ for }\; h\in {\mathfrak h},
 m \ge 0\}.
\end{equation*}
Then $K(sl_2,k)$ which is the space of highest weight vectors
with highest weight $0$ for $\wh$
is the commutant of $M(k)$ in $L(k,0)$
and is called the parafermion vertex operator algebra associated
to the irreducible highest weight module $L(k,0)$ for
$\widehat{sl_2}.$ The Virasoro element of $K(sl_2,k)$ is given by
$$\omega =\omega_{\mraff} - \omega_{\gamma}=\frac{1}{2k(k+2)} \Big(-kh(-2)\1-h(-1)^2\1+2k
e(-1)f(-1)\1 \Big)
$$
with central charge $\frac{2(k-1)}{k+2}$,
where we still use $\omega_{\mraff}, \omega_{\gamma}$ to
denote their images in $L(k,0)$. We denote $K(sl_2,k)$ by $K_0$.

Set \begin{equation*}\label{eq:W3}
\begin{split}
W^3 &= k^2 h(-3)\1 + 3 k h(-2)h(-1)\1 +
2h(-1)^3\1 - 6k h(-1)e(-1)f(-1)\1 \\
& \quad + 3 k^2e(-2)f(-1)\1 - 3 k^2e(-1)f(-2)\1
\end{split}
\end{equation*}
in $V(k,0)$, and also denote its image in $L(k,0)$ by $W^{3}$. It was proved in \cite{DLY2}(cf.\cite{DLWY}, \cite{DW1}) that the parafermion vertex operator algebra $K_0$ is simple and is generated by $\omega$ and $W^{3}$. If $k\geq3$, the parafermion vertex operator algebra $K_0$ in fact is generated by $W^{3}$. The irreducible $K_0$-modules $M^{i,j}$ for $0\leq i\leq k, 0\leq j\leq k-1$ were constructed in \cite{DLY2}. Note that $K_0=M^{0,0}$. It was also proved that $M^{i,j}\cong M^{k-i,k-i+j}$ as $K_0$-module in
\cite[Theorem 4.4]{DLY2}.   Theorem 8.2 in \cite{ALY1} showed that the $\frac{k(k+1)}{2}$ irreducible $K_0$-modules $M^{i,j}$ for $1\leq i\leq k, 0\leq j\leq i-1$ constructed in \cite{DLY2} form a complete set of isomorphism classes of irreducible $K_0$-modules. Moreover, $K_0$ is $C_2$-cofinite \cite{ALY1} and rational \cite{ALY2}.

 Let $L(k,i)$ for $0\leq i\leq k$ be the irreducible modules for the rational vertex operator algebra $L(k,0)$ with the top level $U^{i}=\bigoplus_{j=0}^{i}\mathbb{C}v^{i,j}$ which is an $(i+1)$-dimensional irreducible module of the simple Lie algebra $\C h(0)\oplus\C e(0)\oplus \C f(0)\cong sl_2$ \cite{DLY2}. The following result was due to \cite{DLY2}.

\begin{lem}\label{lem:lowest}
The operator $o(\omega)=\omega_{1}$ acts on $v^{i,j}, \ 0\leq i\leq k,\ 0\leq j\leq i$ as follows:

\begin{eqnarray}\label{lowest}
o(\omega)v^{i,j}=\frac{1}{2k(k+2)}\Big(k(i-2j)-(i-2j)^{2}+2kj(i-j+1)\Big)v^{i,j}.
 \end{eqnarray}
\end{lem}

Let $\sigma$ be an automorphism of Lie algebra $sl_2$ defined by $\sigma(h)=-h, \ \sigma(e)=f,\ \sigma(f)=e$. $\sigma$ can be lifted to an automorphism $\sigma$ of the vertex operator algebra $V(k,0)$ of order 2 in the following way:
$$\sigma(x_{1}(-n_{1})\cdots x_{s}(-n_{s})\1)=\sigma(x_{1})(-n_{1})\cdots \sigma(x_{s})(-n_{s})\1$$
for $x_{i}\in sl_2$ and $n_{i}>0$. Then $\sigma$ induces an automorphism of $L(k,0)$ as $\sigma$ preserves the unique maximal ideal $\mathcal{J}$, and the Virasoro element $\omega_{\gamma}$ is invariant under $\sigma$. Thus $\sigma$ induces an automorphism of the parafermion vertex operator algebra $K_{0}$. In fact, $\sigma(\omega)=\omega,\ \sigma(W^{3})=-W^{3}$.

\begin{lem}\cite{DLY2}\label{lem:auto}
If $k\geq 3$, the automorhism group Aut$K_{0}=\langle \sigma \rangle$ is of order 2.
\end{lem}
\begin{rmk} If $k=1$, $K_{0}=\C \1$. If $k=2$, $K_0$ is generated by $\omega$. Thus the automorphism group Aut$K_0=\{1\}$ is trivial for $k=1$ and $k=2$. Therefore, by Lemma \ref{lem:auto},  we  only need to consider the orbifold of parafermion vertex operator algebra under the automorphism $\sigma$ for $k\geq 3$.
\end{rmk}

Let $K_{0}^{\sigma}$ be the orbifold vertex operator algebra under the automorphism $\sigma$. We will construct and classify irreducible modules of $K_{0}^{\sigma}$ for $k\geq 3$.

\section{Classification and construction of irreducible modules of $K_{0}^{\sigma}$
}\label{Sect:maximal-ideal-tI}\def\theequation{3.\arabic{equation}}
In this section,  we first recall the definition of weak $g$-twisted module, $g$-twisted module and admissible $g$-twisted module following \cite{DLM3, DLM4}. Then for $k\geq 3$, we construct the irreducible $\sigma$-twisted modules of parafermion vertex operator algebra $K_{0}$. Furthermore, we determine all the irreducible modules of the orbifold vertex operator algebra $K_{0}^{\sigma}$.

Let $\left(V,Y,1,\omega\right)$ be a vertex operator algebra (see
\cite{FLM}, \cite{LL}) and $g$ an automorphism of $V$ with finite order $T$.  Let $W\left\{ z\right\} $
denote the space of $W$-valued formal series in arbitrary complex
powers of $z$ for a vector space $W$. Denote the decomposition of $V$ into eigenspaces with respect to the action of $g$ by $$V=\bigoplus_{r\in\Z}V^{r},$$
where $V^{r}=\{v\in V|\ gv=e^{-\frac{2\pi ir}{T}}v\},$ $i=\sqrt{-1}$.

\begin{definition}A \emph{weak $g$-twisted $V$-module} $M$ is
a vector space with a linear map
\[
Y_{M}:V\to\left(\text{End}M\right)\{z\}
\]

\[
v\mapsto Y_{M}\left(v,z\right)=\sum_{n\in\mathbb{Q}}v_{n}z^{-n-1}\ \left(v_{n}\in\mbox{End}M\right)
\]
which satisfies the following conditions for $0\leq r\leq T-1$, $u\in V^{r}\ ,v\in V, w\in M$:

\[ Y_{M}\left(u,z\right)=\sum_{n\in\frac{r}{T}+\mathbb{Z}}u_{n}z^{-n-1}
\]

\[
u_{n}w=0\ {\rm for} \ n\gg0,
\]

\[
Y_{M}\left(\mathbf{1},z\right)=Id_{M},
\]

\[
z_{0}^{-1}\text{\ensuremath{\delta}}\left(\frac{z_{1}-z_{2}}{z_{0}}\right)Y_{M}\left(u,z_{1}\right)
Y_{M}\left(v,z_{2}\right)-z_{0}^{-1}\delta\left(\frac{z_{2}-z_{1}}{-z_{0}}\right)Y_{M}\left(v,z_{2}\right)Y_{M}\left(u,z_{1}\right)
\]

\[
=z_{1}^{-1}\left(\frac{z_{2}+z_{0}}{z_{1}}\right)^{\frac{r}{T}}\delta\left(\frac{z_{2}+z_{0}}{z_{1}}\right)Y_{M}\left(Y\left(u,z_{0}\right)v,z_{2}\right),
\]
 where $\delta\left(z\right)=\sum_{n\in\mathbb{Z}}z^{n}$. \end{definition}
 The following identities are the consequences of the twisted-Jacobi identity \cite{DLM3} (see also \cite{Ab}, \cite{DJ}).
 \begin{eqnarray}[u_{m+\frac{r}{T}},v_{n+\frac{s}{T}}]=\sum_{i=0}^{\infty}\binom{m+\frac{r}{T}}{i} (u_{i}v)_{m+n+\frac{r+s}{T}-i},\label{eq:3.1.}\end{eqnarray}
 \begin{eqnarray}\sum_{i\geq 0}\binom{\frac{r}{T}}{i}(u_{m+i}v)_{n+\frac{r+s}{T}-i}=\sum_{i\geq 0}(-1)^{i}\binom{m}{i}(u_{m+\frac{r}{T}-i}v_{n+\frac{s}{T}+i}-(-1)^{m}v_{m+n+\frac{s}{T}-i}u_{\frac{r}{T}+i}),\label{eq:3.2.}\end{eqnarray}
 where $u\in V^{r}, \ v\in V^{s},\ m,n\in \Z$.

\begin{definition}

A \emph{$g$-twisted $V$-module} is a weak $g$-twisted $V$-module\emph{
}$M$ which carries a $\mathbb{C}$-grading $M=\bigoplus_{\lambda\in\mathbb{C}}M_{\lambda},$
where $M_{\lambda}=\{w\in M|L(0)w=\lambda w\}$ and $L(0)$ is one of the coefficient operators of $Y(\omega,z)=\sum_{n\in\mathbb{Z}}L(n)z^{-n-2}.$
Moreover we require
that $\dim M_{\lambda}$ is finite and for fixed $\lambda,$ $M_{\lambda+\frac{n}{T}}=0$
for all small enough integers $n.$

\end{definition}

\begin{definition}An \emph{admissible $g$-twisted $V$-module} $M=\oplus_{n\in\frac{1}{T}\mathbb{Z}_{+}}M\left(n\right)$
is a $\frac{1}{T}\mathbb{Z}_{+}$-graded weak $g$-twisted module
such that $u_{m}M\left(n\right)\subset M\left(\mbox{wt}u-m-1+n\right)$
for homogeneous $u\in V$ and $m,n\in\frac{1}{T}\mathbb{Z}.$ $ $

\end{definition}


\begin{definition}A vertex operator algebra $V$ is called \emph{$g$-rational}
if the admissible $g$-twisted module category is semisimple. \end{definition}
\begin{rmk} Since $K_0$ is a rational vertex operator algebra, $K_{0}^{\sigma}$ is $C_2$-cofinite and rational \cite{M2}, \cite{CM}, \cite{CKLR}. It follows from \cite{DRX} that all the irreducible modules of $K_{0}^{\sigma}$ come from modules and $\sigma$-twisted modules of $K_{0}$.
\end{rmk}

We first have the following result.

\begin{prop}\label{prop:twisted1}
If $k=2n+1$ for $n\geq 1$, there are at most  $\frac{k-1}{2}+1$ inequivalent irreducible $\sigma$-twisted modules of $K_0$. If $k=2n$ for $n\geq 2$, there are at most  $\frac{k}{2}+2$ inequivalent  irreducible $\sigma$-twisted modules of $K_0$.
\end{prop}

\begin{proof} Let $(W,Y_{W})$ be an irreducible module of $K_0$, for $u\in K_0$, $v\in W$, we define a linear map:

$$Y^{\sigma}: K_0\rightarrow (\mbox{End}W)[[x,x^{-1}]]$$
by $$Y^{\sigma}_{W}(u,x)v=Y_{W}(\sigma^{-1}u,x)v.$$
From \cite{DLM4}, we know that $(W,Y^{\sigma}_{W})$ is still an irreducible module of $K_0$, if $(W,Y_{W})\cong (W,Y^{\sigma}_{W})$, we say $K_0$-module $W$ is $\sigma$-stable, and denote it by $W^{\sigma}$. From \cite{ALY1} (also see \cite{DLY2}), we know that the $\frac{k(k+1)}{2}$ inequivalent irreducible $K_0$-modules $M^{i,j}$ for $1\leq i\leq k, 0\leq j\leq i-1$ exhaust all the irreducible modules of $K_0$, and the top level of $M^{i,j}$ is a one dimensional space spanned by $v^{i,j}$
with the lowest weight $\frac{1}{2k(k+2)}\Big(k(i-2j)-(i-2j)^{2}+2kj(i-j+1)\Big)$(see Lemma \ref{lem:lowest}), that is,
\begin{eqnarray*}
o(\omega)v^{i,j}=\frac{1}{2k(k+2)}\Big(k(i-2j)-(i-2j)^{2}+2kj(i-j+1)\Big)v^{i,j},
 \end{eqnarray*}
 where $\omega$ is the Virasoro element of parafermion vertex operator algebra $K_0$.
 Thus by analyzing its lowest weights, we see that if $k=2n+1$, $n\geq 1$, $M^{i,j}$ for $(i,j)=(i,\frac{i}{2})$, $i=2,4,6,\cdots,2n$, together with $(i,j)=(2n+1,0)$ are $\sigma$-stable modules, and  there are totally $n+1$, i.e., $\frac{k-1}{2}+1$ inequivalent $\sigma$-stable $K_0$-modules. If $k=2n$, $n\geq 2$, then for $(i,j)=(i,\frac{i}{2})$,
 $i=2,4,6,\cdots,2n$, together with $(i,j)=(n,0)$ and $(i,j)=(2n,0)$, $M^{i,j}$ are $\sigma$-stable modules, and  there are totally $n+2$, i.e., $\frac{k}{2}+2$ inequivalent $\sigma$-stable $K_0$-modules. More precisely, if $k=2n+1$, $n\geq 1$, we have
 $$M^{i,j}\cong (M^{i,j})^{\sigma} \ \ \mbox{for} \ (i,j)=(i,\frac{i}{2}), \ i=2,4,6,\cdots,2n, \mbox{and} \ (i,j)=(2n+1,0).$$
 If $k=2n$, $n\geq 2$,  we have
 $$M^{i,j}\cong (M^{i,j})^{\sigma} \ \ \mbox{for} \ (i,j)=(i,\frac{i}{2}),\
 i=2,4,6,\cdots,2n, \ (i,j)=(n,0) \ \mbox{and} \ (i,j)=(2n,0).$$

 Moreover, we have if $k=2n+1$, $n\geq 1$,
 $$(M^{i,0})^{\sigma}\cong M^{k-i,0}\ \ \mbox{for} \ 1\leq i\leq k-1,$$
 and
 $$(M^{i,j})^{\sigma}\cong M^{i,i-j}\ \ \mbox{for} \ 3\leq i\leq k,\ 1\leq j\leq i-1,\ j\neq\frac{i}{2}.$$
If $k=2n$, $n\geq 2$,  we have
$$(M^{i,0})^{\sigma}\cong M^{k-i,0}\ \ \mbox{for} \ 1\leq i\leq k-1, \ i\neq \frac{k}{2},$$
 and
 $$(M^{i,j})^{\sigma}\cong M^{i,i-j}\ \ \mbox{for} \ 3\leq i\leq k,\ 1\leq j\leq i-1,\ j\neq\frac{i}{2}.$$
 Thus, from the Theorem 10.2 in \cite{DLM4}, we obtain the proposition.

\end{proof}

Now for $k=2n+1$, $n\geq 1$, we construct $\frac{k-1}{2}+1$ inequivalent $\sigma$-twisted irreducible modules of $K_0$, and for $k=2n$, $n\geq 2$, we construct $\frac{k}{2}+2$ inequivalent $\sigma$-twisted irreducible modules of $K_0$.  We will first construct irreducible $\sigma$-twisted modules for the vertex operator algebra $L(k,0)$.

Recall that $\{ h, e, f\}$
is a standard Chevalley basis of $sl_2$ with brackets $[h,e] = 2e$, $[h,f] = -2f$,
$[e,f] = h$. Set $$h^{'}=e+f, \ e^{'}=\frac{1}{2}(h-e+f),\ f^{'}=\frac{1}{2}(h+e-f).$$ Then $\{ h^{'}, e^{'}, f^{'}\}$ is a $sl_2$-triple. Let $h^{''}=\frac{1}{4}h^{'}=\frac{1}{4}(e+f)$.
We have
$$L_{\mraff}(n)h^{''}=\delta_{n,0}h^{''},\ h^{''}(n)h^{''}=\frac{1}{8}\delta_{n,1}k, \mbox{for}\ n\in {\Z}_{+},$$
and $$h^{''}(0)e^{'}=\frac{1}{2}e^{'},\ h^{''}(0)f^{'}=-\frac{1}{2}f^{'},\  h^{''}(0)h^{''}=0,\ e^{'}(0)f^{'}=4h^{''},$$
where $L_{\mraff}(n)=\omega_{\mraff}(n+1).$
This shows that $h^{''}(0)$ semisimply acts on $L(k,0)$ with rational eigenvalues. From \cite{L2}, we have that $e^{2\pi ih^{''}(0)}$ is an automorphism of $L(k,0)$. Moreover, we have
$$e^{2\pi ih^{''}(0)}(h^{'})=h^{'},\ e^{2\pi ih^{''}(0)}(e^{'})=-e^{'},\  e^{2\pi ih^{''}(0)}(f^{'})=-f^{'},$$
thus $e^{2\pi ih^{''}(0)}=\sigma$.
Let $$\Delta(h^{''},z)=z^{h^{''}(0)}\mbox{exp}(\sum_{k=1}^{\infty}\frac{h^{''}(k)}{-k}(-z)^{-k}).$$
Recall that $L(k,i)$ for $0\leq i\leq k$ are all the irreducible modules for the rational vertex operator algebra $L(k,0)$. From \cite{L2}, we have the following result.

\begin{lem}\label{lem:twisted}
For $0\leq i\leq k$, $(\overline{L(k,i)}, Y_{\sigma}(\cdot,z))=(L(k,i),Y(\Delta(h^{''},z)\cdot,z))$ are irreducible $\sigma$-twisted $L(k,0)$-modules.
\end{lem}

Notice that the weight 2 subspace of vertex operator algebra $L(k,0)$ is 9 dimensional for $k\geq 2$. Let
$$\xi_{1}=e(-2)\1+f(-2)\1,\ \ \xi_{2}=-\frac{1}{2}h(-1)^{2}\1+e(-1)^{2}\1+f(-1)^{2}\1,$$
$$\xi_{3}=-\frac{1}{2}h(-2)\1+\frac{1}{4}h(-1)^{2}\1+e(-1)f(-1)\1,$$

\begin{eqnarray*}
\xi_{4}&=&-\frac{1}{2}h(-2)\1-e(-2)\1-f(-2)\1-\frac{1}{2}h(-1)^{2}\1-\frac{1}{2}e(-1)^{2}\1-\frac{1}{2}f(-1)^{2}\1\nonumber\\
 && +h(-1)e(-1)\1-h(-1)f(-1)\1+e(-1)f(-1)\1,
\end{eqnarray*}

\begin{eqnarray*}
\xi_{5}&=&-\frac{1}{2}h(-2)\1+e(-2)\1+f(-2)\1-\frac{1}{2}h(-1)^{2}\1-\frac{1}{2}e(-1)^{2}\1-\frac{1}{2}f(-1)^{2}\1\nonumber\\
 && -h(-1)e(-1)\1+h(-1)f(-1)\1+e(-1)f(-1)\1,
\end{eqnarray*}

$$\xi_{6}=h(-2)\1-e(-2)\1+f(-2)\1,$$

$$\xi_{7}=-h(-2)\1-e(-1)^{2}\1+f(-1)^{2}\1+h(-1)e(-1)\1+h(-1)f(-1)\1,$$

$$\xi_{8}=-h(-2)\1-e(-2)\1+f(-2)\1,$$

$$\xi_{9}=h(-2)\1+e(-1)^{2}\1-f(-1)^{2}\1+h(-1)e(-1)\1+h(-1)f(-1)\1,$$
which are the eigenvectors of $h^{''}(0)$ on the weight 2 subspace of $L(k,0)$ for $k\geq 2$. We have $$h^{''}(0)\xi_{1}=0,\ h^{''}(0)\xi_{2}=0,\ h^{''}(0)\xi_{3}=0,$$
$$h^{''}(0)\xi_{4}=\xi_{4},\ h^{''}(0)\xi_{5}=-\xi_{5},$$
$$h^{''}(0)\xi_{6}=\frac{1}{2}\xi_{6},\ h^{''}(0)\xi_{7}=\frac{1}{2}\xi_{7},$$
$$h^{''}(0)\xi_{8}=-\frac{1}{2}\xi_{8},\ h^{''}(0)\xi_{9}=-\frac{1}{2}\xi_{9}.$$
Moreover, recall that $$\omega =\omega_{\mraff} - \omega_{\gamma}=\frac{1}{2k(k+2)} \Big(-kh(-2)\1-h(-1)^2\1+2k
e(-1)f(-1)\1 \Big),$$ thus we have $$\omega=\frac{1}{8k}\xi_{2}+\frac{3k-2}{4k(k+2)}\xi_{3}+\frac{1}{8k}\xi_{4}+\frac{1}{8k}\xi_{5}.$$
Notice that $$h^{''}(1)\omega=\frac{k-1}{k}h^{''}(-1)\1\ \mbox{and} \ h^{''}(1)^{2}\omega=\frac{k-1}{8}\1.$$
Using these facts, we can obtain the following lemma by a straightforward calculation.

\begin{lem}\label{lem:cal}
\begin{eqnarray}\Delta(h^{''},z)\omega&=&\frac{1}{8k}\xi_{2}+\frac{3k-2}{4k(k+2)}\xi_{3}+z\frac{1}{8k}\xi_{4}+z^{-1}\frac{1}{8k}\xi_{5}\nonumber\\
 && +z^{-1}\frac{k-1}{k}h^{''}(-1)\1+z^{-2}\frac{k-1}{16}\1,\label{eq:3.2}
 \end{eqnarray}
 \begin{eqnarray*}\Delta(h^{''},z)\omega_{\mraff}=L_{\mraff}(-2)\1+z^{-1}h^{''}(-1)\1+z^{-2}\frac{k}{16}\1,
 \end{eqnarray*}
\begin{eqnarray*} Y_{\sigma}(h^{''},z)=Y(h^{''}+\frac{k}{8}z^{-1},z), \end{eqnarray*}

\begin{eqnarray} Y_{\sigma}(h^{'},z)=Y(h^{'}+\frac{k}{2}z^{-1},z), \label{eq:3.3.} \end{eqnarray}

$$Y_{\sigma}(e^{'},z)=z^{\frac{1}{2}}Y(e^{'},z),\ Y_{\sigma}(f^{'},z)=z^{-\frac{1}{2}}Y(f^{'},z).$$
\end{lem}

In the following part of the paper, for $u\in L(k,0)$ such that $\sigma(u)=e^{-\pi ri}$, $i=\sqrt{-1}$, $r\in\Z$, we use the notation $u_{n}$ and $u(n)$ respectively to distinguish the action of the operators in $L(k,0)$ on $\sigma$-twisted modules and untwisted modules, i.e., $u_{n}$ and $u(n)$ are as follows $$Y_{\sigma}(u,z)=\sum_{n\in \Z+\frac{r}{2}}u_{n}z^{-n-1},\ Y(u,z)=\sum_{n\in \Z}u(n)z^{-n-1}.$$
Recall that the top level $U^{i}=\bigoplus_{j=0}^{i}\mathbb{C}v^{i,j}$ of $L(k,i)$ for $0\leq i\leq k$ is an $(i+1)$-dimensional irreducible module for $\C h(0)\oplus \C e(0)\oplus \C f(0)\cong sl_2$. Let
\begin{eqnarray}
\eta=\sum_{j=0}^{i}(-1)^{j}v^{i,j}, \label{eq:3.3'}
\end{eqnarray} then $\eta$ is the lowest weight vector with weight $-i$ in $(i+1)$-dimensional irreducible module for $\C h^{'}(0)\oplus \C e^{'}(0)\oplus \C f^{'}(0)\cong sl_2$, that is, $f^{'}(0)\eta=0$ and $h^{'}(0)\eta=-i\eta$. Notice that $\sigma(\omega)=\omega$, by (\ref{eq:3.2}) in Lemma \ref{lem:cal}, we have
 \begin{eqnarray*}L(0)\eta=\omega_{1}\eta&=&\Big(\frac{1}{8k}\xi_{5}(0)+\frac{1}{8k}\xi_{2}(1)+\frac{3k-2}{4k(k+2)}\xi_{3}(1)+\frac{k-1}{k}h^{''}(0)+\frac{k-1}{16}\1\Big) \eta\nonumber\\
 &=&\Big(\frac{k-2}{8k(k+2)}h(0)^{2}+\frac{1}{8k}(e(0)^{2}+f(0)^{2})+\frac{3k-2}{8k(k+2)}h(0)\nonumber\\
 &&
 +\frac{3k-2}{4k(k+2)}f(0)e(0)+\frac{k-1}{4k}(e(0)+f(0))+\frac{k-1}{16}\Big)\eta
 \nonumber\\
 &=&\Big(\frac{i(i-k)}{4(k+2)}+\frac{k-1}{16}\Big)\eta,
 \end{eqnarray*}
 where we used the action of $sl_2$ on $(i+1)$-dimensional irreducible module $U^{i}=\bigoplus_{j=0}^{i}\mathbb{C}v^{i,j}$, that is, $$h(0)v^{i,j} = (i-2j)v^{i,j},$$
$$e(0)v^{i,0} = 0,\  e(0)v^{i,j} = (i-j+1)v^{i,j-1} \ \mbox{for} \ 1 \leq j \leq
i,$$

$$f(0)v^{i,i} = 0, \ f(0)v^{i,j} = (j+1)v^{i,j+1} \ \mbox{for}\ 0 \leq j \leq
i-1,$$

$$a(n)v^{i,j} = 0  \ \mbox{for}\ a \in \{h,e,f\}, n \geq 1.$$
Thus we have:
\begin{lem}\label{lem:weight}
For the positive integer $k\geq 3$, $0\leq i\leq k$,
\begin{eqnarray}L(0)\eta=\Big(\frac{i(i-k)}{4(k+2)}+\frac{k-1}{16}\Big)\eta. \label{eq:3.3}
\end{eqnarray}
\end{lem}

\begin{lem}\label{lem:lowest1}
The vector $(e-f)_{-\frac{1}{2}}.\eta$ is another lowest weight vector in $\sigma$-twisted module $\overline{L(k,\frac{k}{2})}$ of parafermion vertex operator algebra $K_0$ besides $\eta$.
\end{lem}
\begin{proof}
From \cite{DLY2}, we know that $\omega, W^{3}, W^{4}, W^{5}$ are the strong generators of the parafermion vertex operator algebra $K_0$. It is easy to see that $\omega_{2}.(e-f)_{-\frac{1}{2}}.\eta=0$, $W^{4}_{4}.(e-f)_{-\frac{1}{2}}.\eta=0$. Now we prove
${W^{3}_{\frac{5}{2}}}.(e-f)_{-\frac{1}{2}}.\eta=0$. Recall that \begin{equation*}
\begin{split}
W^3 &= k^2 h(-3)\1 + 3 k h(-2)h(-1)\1 +
2h(-1)^3\1 - 6k h(-1)e(-1)f(-1)\1 \\
& \quad + 3 k^2e(-2)f(-1)\1 - 3 k^2e(-1)f(-2)\1.
\end{split}
\end{equation*}
and
 \begin{equation*}
 3k h(-2)h(-1)\1- 6k h(-1)e(-1)f(-1)\1
=-3k\Big(h(-1)e(-1)f(-1)\1+h(-1)f(-1)e(-1)\1\Big).
\end{equation*}
By using (\ref{eq:3.2.}), we have
\begin{equation}(h(-3)\1)_{\frac{5}{2}}=\frac{15}{8}h_{\frac{1}{2}}, \ \ (h(-1)^{3}\1)_{\frac{5}{2}}=2h_{-\frac{1}{2}}(h_{\frac{1}{2}})^{2},\label{eq:3.10.}\end{equation}

\begin{equation}
\begin{split}
& \Big(h(-1)e(-1)f(-1)\1+h(-1)f(-1)e(-1)\1\Big)_{\frac{5}{2}}\\
  &= \frac{1}{2}\Big(h(-1)(e+f)(-1)(e+f)(-1)\1-h(-1)(e-f)(-1)(e-f)(-1)\1\Big)_{\frac{5}{2}} \\
&= \frac{1}{2}\Big((e+f)_{-1}(e+f)_{1}h_{\frac{1}{2}}-(e-f)_{-1}(e-f)_{1}h_{\frac{1}{2}}+\frac{1}{2}h_{\frac{1}{2}}\Big),\label{eq:3.11.}
\end{split}
\end{equation}
and
\begin{equation}
\begin{split}
& \Big(e(-2)f(-1)\1-e(-1)f(-2)\1\Big)_{\frac{5}{2}}\\
 &= \Big(f(-1)e(-2)\1-f(-2)e(-1)\1\Big)_{\frac{5}{2}} \\
  &= \frac{1}{2}\Big((e+f)(-1)(e-f)(-2)\1-(e+f)(-2)(e-f)(-1)\1\Big)_{\frac{5}{2}} \\
&= -\frac{1}{4}\Big((e+f)_{0}(e-f)_{\frac{1}{2}}-2h_{\frac{1}{2}}\Big).\label{eq:3.12.}
\end{split}
\end{equation}
By using (\ref{eq:3.1.}),  (\ref{eq:3.3.}) and noticing that $i=\frac{k}{2}$,
we have
$$h_{\frac{1}{2}}.(e-f)_{-\frac{1}{2}}\eta=2(e+f)_{0}\eta=2(-i+\frac{k}{2})\eta=0.$$
$$(e+f)_{0}(e-f)_{\frac{1}{2}}.(e-f)_{-\frac{1}{2}}\eta=-k(e+f)_{0}\eta=-k(-i+\frac{k}{2})\eta=0.$$
Together with (\ref{eq:3.10.}), (\ref{eq:3.11.}), (\ref{eq:3.12.}),
we deduce that ${W^{3}_{\frac{5}{2}}}_{\cdot}(e-f)_{-\frac{1}{2}}.\eta=0$. Similarly, by a straightforward calculation, we can obtain that ${W^{5}_{\frac{9}{2}}}.(e-f)_{-\frac{1}{2}}.\eta=0$. Thus $(e-f)_{-\frac{1}{2}}.\eta$ is another lowest weight vector in $\sigma$-twisted module $\overline{L(k,\frac{k}{2})}$ of $K_0$ besides $\eta$.
\end{proof}

 Now we are in a position to state the following theorem.
\begin{thm}\label{thm:construct}
For $0\leq i\leq k$, let $W(k,i)(\subseteq \overline{L(k,i)})$ be the $\sigma$-twisted modules of the parafermion vertex operator algebra $K_0$ generated by $\eta$ in (\ref{eq:3.3'}). If $k=2n+1$, $n\geq 1$, then $W(k,i)$ for $0\leq i\leq \frac{k-1}{2}$ are $\frac{k-1}{2}+1$ irreducible $\sigma$-twisted modules of  $K_0$. If $k=2n$, $n\geq 2$, then $W(k,i)$ for $0\leq i\leq \frac{k}{2}$ are $\frac{k}{2}+1$ irreducible $\sigma$-twisted modules of $K_0$. Let $\widetilde{W(k,\frac{k}{2})}$ be the $\sigma$-twisted module of $K_0$ generated by vector $(e-f)_{-\frac{1}{2}}\eta$, then $\widetilde{W(k,\frac{k}{2})}$ is an irreducible $\sigma$-twisted module of $K_0$, that is, there are $\frac{k}{2}+2$ irreducible $\sigma$-twisted modules of $K_0$ for $k=2n$.
\end{thm}
\begin{proof} First we consider the case $k=2n+1$. From (\ref{eq:3.3}), we see that for $0\leq i\leq \frac{k-1}{2}$, $W(k,i)$ are $\frac{k-1}{2}+1$ inequivalent $\sigma$-twisted modules of $K_0$. Now we prove that $W(k,i)$ for $0\leq i\leq \frac{k-1}{2}$ are irreducible $\sigma$-twisted modules of $K_0$. If $W(k,i)$ is not irreducible for some $i$, the lowest weight $\mu$ of the maximal proper submodule is different from the weights $\Big(\frac{i(i-k)}{4(k+2)}+\frac{k-1}{16}\Big)$, $0\leq i\leq \frac{k-1}{2}$. From \cite{DLM4}, $K_0$ has an irreducible $\sigma$-twisted module with the lowest weight $\mu$. This implies that there are at least $\frac{k-1}{2}+2$ inequivalent irreducible $\sigma$-twisted modules of $K_0$, which contradicts to the result in Proposition \ref{prop:twisted1}. If $k=2n$, from Lemma \ref{lem:lowest1}, we know that $(e-f)_{-\frac{1}{2}}.\eta$ is another lowest weight vector in $\sigma$-twisted module $\overline{L(k,\frac{k}{2})}$ of $K_0$ besides $\eta$. Let $\widetilde{W(k,\frac{k}{2})}$ be the $\sigma$-twisted module of $K_0$ generated by vector $(e-f)_{-\frac{1}{2}}.\eta$, from the result in Proposition \ref{prop:twisted1} that there are at most $\frac{k}{2}+2$ inequivalent $\sigma$-twisted irreducible modules of $K_0$, together with the similar arguments as the case $k=2n+1$, we see that $W(k,i)$ for $0\leq i\leq \frac{k}{2}$ and $\widetilde{W(k,\frac{k}{2})}$, are exactly $\frac{k}{2}+2$ irreducible $\sigma$-twisted modules of $K_0$ for $k=2n$.
\end{proof}

\begin{lem}\label{lem:weight.}
(1) For $i=\frac{k}{2}$, ${W^{3}_{\frac{3}{2}}}.(e-f)_{-\frac{1}{2}}.\eta\neq0$.

(2) For $0< i\leq k$, $i\neq\frac{k}{2}$, ${W^{3}_{\frac{3}{2}}}.\eta\neq0$. For $i=0$ and $i=\frac{k}{2}$, ${W^{3}_{\frac{3}{2}}}.\eta=0$.

(3) For $k=4$, $i=\frac{k}{2}=2$, ${W^{3}_{\frac{1}{2}}}.\eta=0$. For $k \neq 4$, $i=\frac{k}{2}$, ${W^{3}_{\frac{1}{2}}}.\eta\neq 0.$

(4) For $k=4$, $i=\frac{k}{2}=2$, ${W^{3}_{-\frac{1}{2}}}.\eta\neq 0.$
\end{lem}
\begin{proof}

 Notice that \begin{equation*}
\begin{split}
W^3 &= k^2 h(-3)\1 + 3 k h(-2)h(-1)\1 +
2h(-1)^3\1 - 6k h(-1)e(-1)f(-1)\1 \\
& \quad + 3 k^2e(-2)f(-1)\1 - 3 k^2e(-1)f(-2)\1,
\end{split}
\end{equation*}
and
 \begin{equation*}
 3k h(-2)h(-1)\1- 6k h(-1)e(-1)f(-1)\1
=-3k\Big(h(-1)e(-1)f(-1)\1+h(-1)f(-1)e(-1)\1\Big).
\end{equation*}

First we prove (1), that is, if $i=\frac{k}{2}$, ${W^{3}_{\frac{3}{2}}}.(e-f)_{-\frac{1}{2}}.\eta\neq0$.
 By using (\ref{eq:3.2.}), we have
\begin{equation}(h(-3)\1)_{\frac{3}{2}}=\frac{3}{8}h_{-\frac{1}{2}}, \ \ (h(-1)^{3}\1)_{\frac{3}{2}}=3(h_{-\frac{1}{2}})^{2}h_{\frac{1}{2}}+\frac{3k}{4}h_{-\frac{1}{2}}+2h_{-\frac{3}{2}}(h_{\frac{1}{2}})^{2}.\label{eq:3.13.}\end{equation}
By using (\ref{eq:3.1.}) and  (\ref{eq:3.3.}), we have
 \begin{equation}
h_{\frac{1}{2}}.(e-f)_{-\frac{1}{2}}\eta=[h_{\frac{1}{2}},(e-f)_{-\frac{1}{2}}]\eta=2(e+f)_{0}\eta=2(-i+\frac{k}{2})\eta=0.\label{eq:3.14.}
\end{equation}
From  (\ref{eq:3.2.}) and (\ref{eq:3.14.}), we get
\begin{equation}
\begin{split}
& \Big(h(-1)e(-1)f(-1)\1+h(-1)f(-1)e(-1)\1\Big)_{\frac{3}{2}}.(e-f)_{-\frac{1}{2}}\eta\\
  &= \frac{1}{2}\Big(h(-1)(e+f)(-1)(e+f)(-1)\1-h(-1)(e-f)(-1)(e-f)(-1)\1\Big)_{\frac{3}{2}}.(e-f)_{-\frac{1}{2}}\eta \\
&= \frac{1}{2}(5+\frac{9k}{4})h_{-\frac{1}{2}}(e-f)_{-\frac{1}{2}}\eta,\label{eq:3.15.}
\end{split}
\end{equation}
and
\begin{equation}
\begin{split}
& \Big(e(-2)f(-1)\1-e(-1)f(-2)\1\Big)_{\frac{3}{2}}.(e-f)_{-\frac{1}{2}}\eta\\
 &= \Big(f(-1)e(-2)\1-f(-2)e(-1)\1\Big)_{\frac{3}{2}}.(e-f)_{-\frac{1}{2}}\eta \\
  &= \frac{1}{2}\Big((e+f)(-1)(e-f)(-2)\1-(e+f)(-2)(e-f)(-1)\1\Big)_{\frac{3}{2}}.(e-f)_{-\frac{1}{2}}\eta \\
&=(\frac{3k}{4}+1)(e+f)_{-1}\eta-\frac{1}{2}h_{-\frac{1}{2}}(e-f)_{-\frac{1}{2}}\eta.\label{eq:3.16.}
\end{split}
\end{equation}
Thus from (\ref{eq:3.13.}), (\ref{eq:3.14.}), (\ref{eq:3.15.}) and (\ref{eq:3.16.}), we have
$${W^{3}_{\frac{3}{2}}}.(e-f)_{-\frac{1}{2}}.\eta=-k(\frac{9k}{2}+6)h_{-\frac{1}{2}}(e-f)_{-\frac{1}{2}}\eta+3k^{2}(\frac{3k}{4}+1)(e+f)_{-1}\eta\neq 0.$$

Next we prove (2), that is, if $0< i\leq k$, $i\neq\frac{k}{2}$, ${W^{3}_{\frac{3}{2}}}.\eta\neq0$. If $i=0$ and $i=\frac{k}{2}$, then ${W^{3}_{\frac{3}{2}}}.\eta=0$.
 By using (\ref{eq:3.2.}), we have
\begin{equation}(h(-3)\1)_{\frac{3}{2}}=\frac{3}{8}h_{-\frac{1}{2}}, \ \ (h(-1)^{3}\1)_{\frac{3}{2}}=3(h_{-\frac{1}{2}})^{2}h_{\frac{1}{2}}+\frac{3k}{4}h_{-\frac{1}{2}}+2h_{-\frac{3}{2}}(h_{\frac{1}{2}})^{2}.\label{eq:3.17.}\end{equation}
We notice that
 \begin{equation}
h_{\frac{1}{2}}.\eta=0.\label{eq:3.18.}
\end{equation}
By using (\ref{eq:3.2.}), we get
\begin{equation}
\begin{split}
& \Big(h(-1)e(-1)f(-1)\1+h(-1)f(-1)e(-1)\1\Big)_{\frac{3}{2}}.\eta\\
  &= \frac{1}{2}\Big(h(-1)(e+f)(-1)(e+f)(-1)\1-h(-1)(e-f)(-1)(e-f)(-1)\1\Big)_{\frac{3}{2}}.\eta \\
&= \frac{1}{2}(1+\frac{k}{4}+(-i+\frac{k}{2})^{2})h_{-\frac{1}{2}}\eta,\label{eq:3.19.}
\end{split}
\end{equation}
and
\begin{equation}
\begin{split}
& \Big(e(-2)f(-1)\1-e(-1)f(-2)\1\Big)_{\frac{3}{2}}.\eta\\
 &= \Big(f(-1)e(-2)\1-f(-2)e(-1)\1\Big)_{\frac{3}{2}}.\eta \\
  &= \frac{1}{2}\Big((e+f)(-1)(e-f)(-2)\1-(e+f)(-2)(e-f)(-1)\1\Big)_{\frac{3}{2}}.\eta \\
&=\frac{1}{4}(-i+\frac{k}{2})(e-f)_{-\frac{1}{2}}\eta.\label{eq:3.20.}
\end{split}
\end{equation}
Thus from (\ref{eq:3.17.}), (\ref{eq:3.18.}), (\ref{eq:3.19.}) and (\ref{eq:3.20.}), we have
\begin{equation}
\begin{split}
{W^{3}_{\frac{3}{2}}}.\eta=-\frac{3k}{2}(-i+\frac{k}{2})^{2}h_{-\frac{1}{2}}\eta+\frac{3k^{2}}{4}(-i+\frac{k}{2})(e-f)_{-\frac{1}{2}}\eta.\label{eq:3.21.}
\end{split}
\end{equation}
If $i=\frac{k}{2}$, it is obvious that ${W^{3}_{\frac{3}{2}}}.\eta=0$ from the identity (\ref{eq:3.21.}). If $i=0$, from (\ref{eq:3.21.}), we see that ${W^{3}_{\frac{3}{2}}}.\eta=-\frac{3k^{2}}{8}(h-e
+f)_{-\frac{1}{2}}\eta=-\frac{3k^{2}}{8}e^{'}_{-\frac{1}{2}}\eta=-\frac{3k^{2}}{8}e^{'}(0)\eta=0$, where for the last equation, we just  notice that $e^{'}(0)\eta=0$ in the irreducible $L(k,0)$-module $L(k,0)$. If
$i\neq 0$ and $i\neq\frac{k}{2}$, it is clear that ${W^{3}_{\frac{3}{2}}}.\eta\neq 0$ from (\ref{eq:3.21.}).

Now we prove (3), that is, if $i=\frac{k}{2}$, $${W^{3}_{\frac{1}{2}}}.\eta=0 \ \ \mbox{for}\ \  k=4,\ \ {W^{3}_{\frac{1}{2}}}.\eta\neq 0 \ \ \mbox{for}\ \  k \neq 4.$$
 By using (\ref{eq:3.2.}), we have
\begin{equation}(h(-3)\1)_{\frac{1}{2}}.\eta=-\frac{1}{8}h_{-\frac{3}{2}}\eta, \ \ (h(-1)^{3}\1)_{\frac{1}{2}}.\eta=(h_{-\frac{1}{2}})^{3}\eta+\frac{3k}{4}h_{-\frac{3}{2}}\eta.\label{eq:3.22.}\end{equation}
By using (\ref{eq:3.2.}), we get
\begin{equation}
\begin{split}
& \Big(h(-1)e(-1)f(-1)\1+h(-1)f(-1)e(-1)\1\Big)_{\frac{1}{2}}.\eta\\
  &= \frac{1}{2}\Big(h(-1)(e+f)(-1)(e+f)(-1)\1-h(-1)(e-f)(-1)(e-f)(-1)\1\Big)_{\frac{1}{2}}.\eta \\
&= \frac{1}{2}(1+\frac{k}{4}+(-i+\frac{k}{2})^{2})h_{-\frac{3}{2}}\eta+(-i+\frac{k}{2})h_{-\frac{1}{2}}(e+f)_{-1}\eta-\frac{1}{2}h_{-\frac{1}{2}}((e-f)_{-\frac{1}{2}})^{2}\eta,\label{eq:3.23.}
\end{split}
\end{equation}
and
\begin{equation}
\begin{split}
& \Big(e(-2)f(-1)\1-e(-1)f(-2)\1\Big)_{\frac{1}{2}}.\eta\\
 &= \Big(f(-1)e(-2)\1-f(-2)e(-1)\1\Big)_{\frac{1}{2}}.\eta \\
  &= \frac{1}{2}\Big((e+f)(-1)(e-f)(-2)\1-(e+f)(-2)(e-f)(-1)\1\Big)_{\frac{1}{2}}.\eta \\
&=-\frac{1}{4}(e+f)_{-1}(e-f)_{-\frac{1}{2}}\eta+\frac{3}{4}(-i+\frac{k}{2})(e-f)_{-\frac{3}{2}}\eta.\label{eq:3.24.}
\end{split}
\end{equation}
Thus from (\ref{eq:3.22.}), (\ref{eq:3.23.}) and (\ref{eq:3.24.}), we have
\begin{equation}
\begin{split}
{W^{3}_{\frac{1}{2}}}.\eta &=(-\frac{k^{2}}{2}-\frac{3k}{2}(-i+\frac{k}{2})^{2})h_{-\frac{3}{2}}\eta+2(h_{-\frac{1}{2}})^{3}\eta \\
& -\frac{3k^{2}}{4}(e+f)_{-1}(e-f)_{-\frac{1}{2}}\eta
+\frac{9k^{2}}{4}(-i+\frac{k}{2})(e-f)_{-\frac{3}{2}}\eta \\
& -3k(-i+\frac{k}{2})h_{-\frac{1}{2}}(e+f)_{-1}\eta+\frac{3k}{2}h_{-\frac{1}{2}}((e-f)_{-\frac{1}{2}})^{2}\eta. \label{eq:3.25.}
\end{split}
\end{equation}
If $i=\frac{k}{2}$, the identity (\ref{eq:3.25.}) becomes
\begin{equation}
\begin{split}
{W^{3}_{\frac{1}{2}}}.\eta &=-\frac{k^{2}}{2}h_{-\frac{3}{2}}\eta+2(h_{-\frac{1}{2}})^{3}\eta-\frac{3k^{2}}{4}(e+f)_{-1}(e-f)_{-\frac{1}{2}}\eta
+\frac{3k}{2}h_{-\frac{1}{2}}((e-f)_{-\frac{1}{2}})^{2}\eta. \label{eq:3.26.}
\end{split}
\end{equation}
Recall that $h=e^{'}+f^{'}$, $e+f=h^{'}$, and $e-f=f^{'}-e^{'}$. Then from Lemma \ref{lem:cal}, we have $e^{'}_{-\frac{3}{2}}=e^{'}(-1)$, $f^{'}_{-\frac{3}{2}}=f^{'}(-2)$, $h^{'}_{-1}=h^{'}(-1)$, $e^{'}_{-\frac{1}{2}}=e^{'}(0)$,
$f^{'}_{-\frac{1}{2}}=f^{'}(-1)$. Thus we write the identity (\ref{eq:3.26.}) in the untwisted case:
\begin{equation}
\begin{split}
{W^{3}_{\frac{1}{2}}}.\eta &=-\frac{k^{2}}{2}(e^{'}(-1)+f^{'}(-2))\eta+2(e^{'}(0)+f^{'}(-1))^{3}\eta \\
&-\frac{3k^{2}}{4}h^{'}(-1)(f^{'}(-1)-e^{'}(0))\eta +\frac{3k}{2}(e^{'}(0)+f^{'}(-1))(f^{'}(-1)-e^{'}(0))^{2}\eta. \label{eq:3.27.}
\end{split}
\end{equation}
Thus we have
\begin{equation}
\begin{split}
e^{'}(1).{W^{3}_{\frac{1}{2}}}.\eta &=(12+6k-\frac{9k^{2}}{4})e^{'}(0)^{2}\eta+(-4+\frac{7k^{2}}{4}-\frac{3k^{3}}{8})h^{'}(-1)\eta \\
&+(2+\frac{3k}{2})(\frac{3k}{2}-6)f^{'}(-1)^{2}\eta,
 \label{eq:3.28.}
\end{split}
\end{equation}

\begin{equation}
\begin{split}
f^{'}(1).{W^{3}_{\frac{1}{2}}}.\eta =0,
 \label{eq:3.29.}
\end{split}
\end{equation}

\begin{equation}
\begin{split}
h^{'}(1).{W^{3}_{\frac{1}{2}}}.\eta &=(\frac{3k^{3}}{2}-7k^{2}+16)e^{'}(0)\eta-(\frac{3k^{3}}{2}-7k^{2}+16)f^{'}(-1)\eta.
 \label{eq:3.30.}
\end{split}
\end{equation}
From (\ref{eq:3.28.}), (\ref{eq:3.29.}), (\ref{eq:3.30.}), we see that if $k=4$, then $e^{'}(1).{W^{3}_{\frac{1}{2}}}.\eta=0, \ f^{'}(1).{W^{3}_{\frac{1}{2}}}.\eta =0, \ h^{'}(1).{W^{3}_{\frac{1}{2}}}.\eta=0$. Thus ${W^{3}_{\frac{1}{2}}}.\eta=0$ in the $L(k,0)$-module $L(k,\frac{k}{2})$, and if $k\neq 4$, ${W^{3}_{\frac{1}{2}}}.\eta\neq 0$ in the $L(k,0)$-module $L(k,\frac{k}{2})$.

Finally, we prove (4), that is, if $k=4$, $i=\frac{k}{2}=2$, $${W^{3}_{-\frac{1}{2}}}.\eta \neq 0.$$
 By using (\ref{eq:3.2.}), we have
\begin{equation}(h(-3)\1)_{-\frac{1}{2}}.\eta=\frac{3}{8}h_{-\frac{5}{2}}\eta, \ \ (h(-1)^{3}\1)_{-\frac{1}{2}}.\eta=\frac{3k}{4}h_{-\frac{5}{2}}\eta+3h_{-\frac{3}{2}}(h_{-\frac{1}{2}})^{2}\eta.\label{eq:3.31.}\end{equation}
By using (\ref{eq:3.2.}), we get
\begin{equation}
\begin{split}
& \Big(h(-1)e(-1)f(-1)\1+h(-1)f(-1)e(-1)\1\Big)_{-\frac{1}{2}}.\eta\\
  &= \frac{1}{2}\Big(h(-1)(e+f)(-1)(e+f)(-1)\1-h(-1)(e-f)(-1)(e-f)(-1)\1\Big)_{-\frac{1}{2}}.\eta \\
&= h_{-\frac{5}{2}}\eta+\frac{1}{2}h_{-\frac{1}{2}}((e+f)_{-1})^{2}-h_{-\frac{1}{2}}(e-f)_{-\frac{3}{2}}(e-f)_{-\frac{1}{2}}-\frac{1}{2}h_{-\frac{3}{2}}((e-f)_{-\frac{1}{2}})^{2}\eta,\label{eq:3.32.}
\end{split}
\end{equation}
and
\begin{equation}
\begin{split}
& \Big(e(-2)f(-1)\1-e(-1)f(-2)\1\Big)_{-\frac{1}{2}}.\eta\\
 &= \Big(f(-1)e(-2)\1-f(-2)e(-1)\1\Big)_{-\frac{1}{2}}.\eta \\
  &= \frac{1}{2}\Big((e+f)(-1)(e-f)(-2)\1-(e+f)(-2)(e-f)(-1)\1\Big)_{-\frac{1}{2}}.\eta \\
&=\frac{1}{4}(e+f)_{-1}(e-f)_{-\frac{3}{2}}\eta-\frac{3}{4}(e+f)_{-2}(e-f)_{-\frac{1}{2}}\eta.\label{eq:3.33.}
\end{split}
\end{equation}
Thus from (\ref{eq:3.31.}), (\ref{eq:3.32.}) and (\ref{eq:3.33.}), we have
\begin{equation}
\begin{split}
{W^{3}_{-\frac{1}{2}}}.\eta &=6h_{-\frac{3}{2}}(h_{-\frac{1}{2}})^{2}\eta+12(e+f)_{-1}(e-f)_{-\frac{3}{2}}\eta-36(e+f)_{-2}(e-f)_{-\frac{1}{2}}\eta \\
&-6h_{-\frac{1}{2}}((e+f)_{-1})^{2}\eta+12h_{-\frac{1}{2}}(e-f)_{-\frac{3}{2}}(e-f)_{-\frac{1}{2}}\eta+6h_{-\frac{3}{2}}((e-f)_{-\frac{1}{2}})^{2}\eta \\
&\neq 0
 \label{eq:3.34.}
\end{split}
\end{equation}
in the $L(4,0)$-module $L(4,2)$.

\end{proof}

Now we are in a position to classify all the irreducible modules of orbifold vertex operator algebra $K_{0}^{\sigma}$.
For $0< i\leq k, i\neq \frac{k}{2}$, set
\begin{eqnarray} u^{k,i,1}=\eta\in W(k,i)(0),\ u^{k,i,2}=W^{3}_{\frac{3}{2}}.\eta\in W(k,i)(\frac{1}{2}).
 \label{eq:3.4}
\end{eqnarray}
For $i=0$, set
\begin{eqnarray} u^{k,0,1}=\eta\in W(k,0)(0),\ u^{k,0,2}=W^{3}_{\frac{1}{2}}.\eta\in W(k,0)(\frac{3}{2}).
 \label{eq:3.4'}
\end{eqnarray}
For $k=4$, $i=\frac{k}{2}=2$, set
\begin{eqnarray} u^{4,2,1}=\eta\in W(4,2)(0),\ u^{4,2,2}=W^{3}_{-\frac{1}{2}}.\eta\in W(4,2)(\frac{5}{2}).
 \label{eq:3.4.'}
\end{eqnarray}
For $k\neq4$, $i=\frac{k}{2}$, set
\begin{eqnarray} u^{k,\frac{k}{2},1}=\eta\in W(k,\frac{k}{2})(0),\ u^{k,\frac{k}{2},2}=W^{3}_{\frac{1}{2}}.\eta\in W(k,\frac{k}{2})(\frac{3}{2}).
 \label{eq:3.4.'}
\end{eqnarray}
For $i=\frac{k}{2}$,
set \begin{eqnarray}\tilde{u}^{k,\frac{k}{2},1}=(e-f)_{-\frac{1}{2}}.\eta \in \widetilde{W(k,\frac{k}{2})}(0), \ \tilde{u}^{k,\frac{k}{2},2}=W^{3}_{\frac{3}{2}}.(e-f)_{-\frac{1}{2}}.\eta\in \widetilde{W(k,\frac{k}{2})}(\frac{1}{2}).\label{eq:3.5}
\end{eqnarray}

Notice that $W(k,i)$ for $k=2n+1, \ n\geq 1, \ 0\leq i\leq \frac{k-1}{2}$, $W(k,i)$ for $k=2n, \ n\geq 2,\ 0\leq i\leq \frac{k}{2},$ and $\widetilde{W(k,\frac{k}{2})}$ in Theorem \ref{thm:construct} can be viewed as $K_{0}^{\sigma}$-modules. By applying the results in \cite{DM1}, we have:

\begin{prop}\label{prop:orbifold1}
Let $W(k,i)^{j}$ be the $K_{0}^{\sigma}$-module generated by $u^{k,i,j}$ for $0\leq i\leq k, \ j=1,2$, and let $\widetilde{W(k,\frac{k}{2})}^{j}$ be the $K_{0}^{\sigma}$-module generated by
$\tilde{u}^{k,\frac{k}{2},j}$ for $j=1,2$. Then $W(k,i)^{j}$ for $0\leq i\leq k, \ j=1,2$ and  $\widetilde{W(k,\frac{k}{2})}^{j}$ for $j=1,2$ are irreducible modules of $K_{0}^{\sigma}$ satisfying
$$ L(0)u^{k,i,1}=\Big(\frac{i(i-k)}{4(k+2)}+\frac{k-1}{16}\Big)u^{k,i,1}, \ L(0)u^{k,i,2}=\Big(\frac{i(i-k)}{4(k+2)}+\frac{k+7}{16}\Big)u^{k,i,2},$$
for $0< i\leq k, i\neq \frac{k}{2}$. And for $i=0$,
$$L(0)u^{k,0,1}=\frac{k-1}{16}u^{k,0,1},\ L(0)u^{k,0,2}=\frac{k+23}{16}u^{k,0,2}.$$
For $k=4$, $i=\frac{k}{2}=2$, $$L(0)u^{4,2,1}=\frac{1}{48}u^{4,2,1},\ L(0)u^{4,2,1}=\frac{121}{48}u^{4,2,2}.$$
For $k\neq 4$, $i=\frac{k}{2}$, $$L(0)u^{k,\frac{k}{2},1}=\Big(-\frac{k^{2}}{16(k+2)}+\frac{k-1}{16}\Big)u^{k,\frac{k}{2},1},\ L(0)u^{k,\frac{k}{2},2}=\Big(-\frac{k^{2}}{16(k+2)}+\frac{k+23}{16}\Big)u^{k,\frac{k}{2},2}.$$
And for $i=\frac{k}{2}$,
$$L(0)\tilde{u}^{k,\frac{k}{2},1}=\Big(-\frac{k^{2}}{16(k+2)}+\frac{k+7}{16}\Big)\tilde{u}^{k,\frac{k}{2},1},\ L(0)\tilde{u}^{k,\frac{k}{2},2}=\Big(-\frac{k^{2}}{16(k+2)}+\frac{k+15}{16}\Big)\tilde{u}^{k,\frac{k}{2},2}.$$
\end{prop}

Recall from \cite[Theorem 4.4]{DLY2} that the top level of the irreducible $K_0$-module $M^{i,j}$ is a one dimensional space $\mathbb{C}v^{i,j}$ for $1\leq i\leq k, 0\leq j\leq i-1$. From Lemma
\ref{lem:lowest}, we know that for $v^{i,j}$ in $M^{i,j}$, $1\leq i\leq k,\  0\leq j\leq i-1$,

\begin{eqnarray*}
o(\omega)v^{i,j}=\frac{1}{2k(k+2)}\Big(k(i-2j)-(i-2j)^{2}+2kj(i-j+1)\Big)v^{i,j}.
 \end{eqnarray*}
Denote $\lambda_{i,j}=\frac{1}{2k(k+2)}\Big(k(i-2j)-(i-2j)^{2}+2kj(i-j+1)\Big)$ for $1\leq i\leq k, \ 0\leq j\leq i-1$.
Then we have:
\begin{prop}\label{prop:orbifold2}
If $k=2n+1$, $n\geq 1$, let $(i,j)=(i,\frac{i}{2}), \ i=2,4,6,\cdots,2n,$ and $(i,j)=(2n+1,0)$.
 If $k=2n$, $n\geq 2$,  let $(i,j)=(i,\frac{i}{2}),\
 i=2,4,6,\cdots,2n,$ $(i,j)=(n,0) \ \mbox{and} \ (i,j)=(2n,0).$ We have
 $$M^{i,j}=(M^{i,j})^{0}\bigoplus (M^{i,j})^{1},$$
 where $(M^{i,j})^{0}$ is an irreducible module of $K_{0}^{\sigma}$ generated by $v^{i,j}$. And if $i\neq k$, $j=\frac{i}{2}$, $(M^{i,j})^{1}$ is an irreducible module of $K_{0}^{\sigma}$ generated by $W^{3}(1).v^{i,j}$ with weight $\lambda_{i,j}+1$. If
$(i,j)=(k,\frac{k}{2})$, $(M^{k,\frac{k}{2}})^{1}$ is an irreducible module of $K_{0}^{\sigma}$ generated by $W^{3}(0).v^{k,\frac{k}{2}}$ with weight $\lambda_{k,\frac{k}{2}}+2$. If $(i,j)=(\frac{k}{2},0)$, $(M^{\frac{k}{2},0})^{1}$ is an irreducible module of $K_{0}^{\sigma}$ generated by $W^{3}(0).v^{\frac{k}{2},0}$ with weight $\lambda_{\frac{k}{2},0}+2$. If $i=k$, $j=0$, $(M^{k,0})^{1}$ is an irreducible module of $K_{0}^{\sigma}$ generated by $W^{3}$ with weight 3.
\end{prop}

\begin{proof}
 If $k=2n+1$, $n\geq 1$, let $(i,j)=(i,\frac{i}{2}), \ i=2,4,6,\cdots,2n,$ and $(i,j)=(2n+1,0)$.
 If $k=2n$, $n\geq 2$,  let $(i,j)=(i,\frac{i}{2}),\
 i=2,4,6,\cdots,2n,$ $(i,j)=(n,0) \ \mbox{and} \ (i,j)=(2n,0).$ From Proposition \ref{prop:twisted1} and \cite{DM1}, we have
 $$M^{i,j}=(M^{i,j})^{0}\bigoplus (M^{i,j})^{1},$$ and $(M^{i,j})^{0}$ is an irreducible module of $K_{0}^{\sigma}$ generated by $v^{i,j}$. Since
\begin{equation}
\begin{split}
W^{3}(1).v^{i,j} &=\Big(-3k(i-2j)+6(i-2j)^{2}-6kj(i-j+1)\Big)h(-1)v^{i,j}\\
& +3k(j+1)\Big(k-2(i-2j)\Big)e(-1)v^{i,j+1}-3k(i-j+1)\Big(2(i-2j)+k\Big)f(-1)v^{i,j-1},
 \label{eq:3.35.}
\end{split}
\end{equation}
taking $j=\frac{i}{2}$, (\ref{eq:3.35.}) becomes
 \begin{equation}
\begin{split}
W^{3}(1).v^{i,j}=-3k(\frac{i}{2}+1)(ih(-1)v^{i,j}-ke(-1)v^{i,j+1}+kf(-1)v^{i,j-1},
 \label{eq:3.36.}
\end{split}
\end{equation}
and \begin{eqnarray}
e(1).W^{3}(1).v^{i,\frac{i}{2}}=-3k(\frac{i}{2}+1)(-i^{2}-2i+2k+k^{2}))v^{i,j-1},\label{eq:3.36.}
\end{eqnarray}
\begin{eqnarray}
f(1).W^{3}(1).v^{i,\frac{i}{2}}=-3k(\frac{i}{2}+1)(i^{2}+2i-2k-k^{2}))v^{i,j+1},\label{eq:3.37.}
\end{eqnarray}
\begin{eqnarray}
h(1).W^{3}(1).v^{i,\frac{i}{2}}=0.\label{eq:3.38.}
\end{eqnarray}
This implies that if $(i,j)=(k,\frac{k}{2})$, $W^{3}(1).v^{k,\frac{k}{2}}=0$, and $i\neq k$, $j=\frac{i}{2}$, $W^{3}(1).v^{i,j}\neq 0$. Thus if $i\neq k$, $j=\frac{i}{2}$, $(M^{i,j})^{1}$ is an irreducible module of $K_{0}^{\sigma}$ generated by $W^{3}(1).v^{i,j}$ with weight $\lambda_{i,j}+1$.
Take $(i,j)=(\frac{k}{2},0)$ in (\ref{eq:3.35.}), then $W^{3}(1).v^{\frac{k}{2},0}=0$.
Now we consider $W^{3}(0).v^{i,j}$ for the case $(i,j)=(k,\frac{k}{2})$ and $(i,j)=(\frac{k}{2},0)$.
Notice that
\begin{equation}
\begin{split}
W^{3}(0).v^{i,j} &=6(i-2j)h(-1)^{2}v^{i,j}+6\Big((i-2j)^{2}-kj(i-j+1)\Big)h(-2)v^{i,j}\\
& -6k(j+1)h(-1)e(-1)v^{i,j+1}-6k(i-j+1)h(-1)f(-1)v^{i,j-1}\\
& -6k(i-2j)e(-1)f(-1)v^{i,j}-6k(j+1)(i-2j-k)e(-2)v^{i,j+1}\\
&-6k(i-j+1)(i-2j+k)f(-2)v^{i,j-1}.
 \label{eq:3.39.}
\end{split}
\end{equation}
Taking $(i,j)=(k,\frac{k}{2})$ in (\ref{eq:3.39.}), we have
\begin{equation}
\begin{split}
W^{3}(0).v^{k,\frac{k}{2}} &=-6k(\frac{k}{2}+1)\Big(\frac{k}{2}h(-2)v^{k,\frac{k}{2}}+h(-1)e(-1)v^{k,\frac{k}{2}+1}\\
& +h(-1)f(-1)v^{k,\frac{k}{2}-1}-ke(-2)v^{k,\frac{k}{2}+1}+kf(-2)v^{k,\frac{k}{2}-1}\Big)\neq 0.
 \label{eq:3.40.}
\end{split}
\end{equation}
Taking $(i,j)=(\frac{k}{2},0)$ in (\ref{eq:3.39.}), we have
\begin{equation}
\begin{split}
W^{3}(0).v^{\frac{k}{2},0} &=3k\Big(h(-1)^{2}v^{\frac{k}{2},0}+\frac{1}{2}kh(-2)v^{\frac{k}{2},0}-2h(-1)e(-1)v^{\frac{k}{2},1}\\
& -ke(-1)f(-1)v^{\frac{k}{2},0}+ke(-2)v^{\frac{k}{2},1}\Big)\neq 0.
 \label{eq:3.41.}
\end{split}
\end{equation}
This shows that if
$(i,j)=(k,\frac{k}{2})$, $(M^{k,\frac{k}{2}})^{1}$ is an irreducible module of $K_{0}^{\sigma}$ generated by $W^{3}(0).v^{k,\frac{k}{2}}$ with weight $\lambda_{k,\frac{k}{2}}+2$, and if $(i,j)=(\frac{k}{2},0)$, $(M^{\frac{k}{2},0})^{1}$ is an irreducible module of $K_{0}^{\sigma}$ generated by $W^{3}(0).v^{\frac{k}{2},0}$ with weight $\lambda_{\frac{k}{2},0}+2$.
If $(i,j)=(k,0)$, we know that $M^{k,0}\cong M^{0,0}=K_0$, thus $(M^{k,0})^{1}$ is an irreducible module of $K_{0}^{\sigma}$ generated by $W^{3}(-1)\1=W^{3}$ with weight 3.

\end{proof}

From \cite{DM1} and the proof of Proposition \ref{prop:twisted1}, we know that if $k=2n+1$, $n\geq 1$, there are totally $\frac{k^{2}-1}{4}$ inequivalent irreducible modules of $K_{0}^{\sigma}$ coming from irreducible modules of $K_{0}$, which are
 $M^{i,0}$ for $1\leq i\leq \frac{k-1}{2},$
  and $M^{i,j}$ for $3\leq i\leq k,$ where if $i=2m$,   $1\leq j\leq m-1$, and if $i=2m+1$,  $1\leq j\leq m$.
If $k=2n$, $n\geq 2$, there are totally $\frac{k^{2}-4}{4}$ inequivalent irreducible modules of $K_{0}^{\sigma}$ coming from irreducible modules of $K_{0}$, which are $M^{i,0}$ for $1\leq i\leq \frac{k-2}{2},$
 and
 $M^{i,j}$ for $3\leq i\leq k,$  where if $i=2m$, $1\leq j\leq m-1$, and if $i=2m+1$,  $1\leq j\leq m$.

\vskip 0.3cm
 From the above discussion together with  Proposition \ref{prop:orbifold1} and  Proposition \ref{prop:orbifold2}, we obtain our final result.

 \begin{thm}\label{thm:orbifold3}
If $k=2n+1$, $n\geq 1$, there are $\frac{(k+1)(k+7)}{4}$ inequivalent irreducible modules of $K_{0}^{\sigma}$.
 If $k=2n$, $n\geq 2$, there are $\frac{(k^{2}+8k+28)}{4}$ inequivalent irreducible modules of $K_{0}^{\sigma}$. Furthermore, the lowest weights of these irreducible modules are given in Proposition \ref{prop:twisted1}, Proposition \ref{prop:orbifold1} and  Proposition \ref{prop:orbifold2}.
\end{thm}

\begin{rmk}  \cite{JLa} has given the level-rank duality for coset orthogonal affine vertex operator algebras. In particular, let $so_{2k}$ be the orthogonal Lie algebra of rank $k$, then
$$C_{K(so_{2k}, 2)}(K(sl_{k},2))\cong K(sl_2,k)^{\sigma},$$
where $C_{K(so_{2k}, 2)}(K(sl_{k},2))$ is the commutant of $K(sl_{k},2)$ in $K(so_{2k}, 2)$.
With the above result, irreducible modules of  $C_{K(so_{2k}, 2)}(K(sl_{k},2))$ are classified and constructed.
\end{rmk}


\begin{thebibliography}{99}

\bibitem{Ab} T. Abe, A $\Z_2$-orbifold model of the symplectic fermionic vertex operator
superalgebra. \emph{Math. Z.} \textbf{255} (4) (2007), 755-792.



\bibitem{ALY1} T. Arakawa, C.H. Lam and H. Yamada, Zhu's algebra, $C_2$-cofiniteness of parafermion
vertex operator algebras, \emph{Adv. Math.}
\textbf{264} (2014), 261-295.

\bibitem{ALY2} T. Arakawa, C.H. Lam and H. Yamada, Parafermion
vertex operator algebras and $W$-algebras,  arXiv:1701.06229.


\bibitem{CM} S. Carnahan and M. Miyamoto, Regularity of fixed-point vertex operator subalgebras, arXiv: 1603.16045v3.

\bibitem{CKLR} T. Creutzig, S. Kanade, A.R. Linshaw and D. Ridout, Schur-Weyl duality for Heisenberg Cosets, arXiv: 1611.00305v1.


\bibitem{DJ} C. Dong and C. Jiang, Representations of the vertex
operator algebra $V_{L_{2}}^{A_{4}}$, {\em J. Algebra} \emph{ }\textbf{377
}(2013), 76-96.






\bibitem{DLM3} C. Dong, H. Li and G. Mason, Twisted representations
of vertex operator algebras, {\em Math. Ann.} \textbf{310} (1998), 571-600.

\bibitem{DLM4} C. Dong, H. Li and G. Mason, Modular-Invariance
of Trace Functions in Orbifold Theory and Generalized Moonshine, {\em Comm.
Math. Phys.} \textbf{214} (2000), 1-56.



\bibitem{DLY2} C. Dong, C.H. Lam and H. Yamada, W-algebras related to
parafermion algebras, \emph{J. Algebra} {\bf 322} (2009), 2366-2403.

\bibitem{DM1} C. Dong and G. Mason, On quantum Galois theory, {\em Duke Math. J.} {\bf 86} (1997), 305-321.

\bibitem{DLWY} C. Dong, C.H. Lam, Q. Wang and H. Yamada, The
structure of parafermion vertex operator algebras,
 \emph{J. Algebra} {\bf 323} (2010), 371-381.


 \bibitem{DR} C. Dong and L. Ren, Representations of the parafermion vertex operator algebras, {\em Adv. Math.} \textbf{315} (2017), 88-101.

 \bibitem{DRX} C. Dong, L. Ren and F. Xu, On orbifold theory, {\em Adv. Math.} \textbf{321} (2017), 1-30.


\bibitem{DW1} C. Dong and Q. Wang, The structure of parafermion vertex operator algebras: general case, {\em Comm. Math. Phys.}
 {\bf 299} (2010), 783-792.

\bibitem{DW2} C. Dong and Q. Wang, On $C_2$-cofiniteness of parafermion vertex operator algebras, {\em J. Algebra}
 {\bf 328} (2011), 420-431.

\bibitem{DW3} C. Dong and Q. Wang, Quantum dimensions and fusion rules for parafermion vertex operator algebras, {\em Proc. Amer. Math. Soc.}
 {\bf 144} (2016), 1483-1492.

 \bibitem{FLM} I. B. Frenkel, J. Lepowsky and A. Meurman, Vertex
operator algebras and the monster, {\em Pure and Applied Math.} Vol. 134,
Academic Press, Massachusetts, 1988.

\bibitem{FZ}
I. B. Frenkel and Y.-C. Zhu, Vertex operator algebras associated to
representations of affine and Virasoro algebras, \emph{Duke Math. J.}
\textbf{66} (1992), 123-168.


\bibitem{JL1} C. Jiang and Z. Lin, The commutant of $L_{\hat{sl_2}}(n,0)$ in the vertex operator algebra $L_{\hat{sl_2}}(n,0)^{\otimes n}$, {\em Adv. Math.} \textbf{301} (2016), 227-257.

\bibitem{JL2} C. Jiang and Z. Lin, Tensor Decompostion, Parafermions, Level-Rank Duality and Reciprocity Law for vertex operator algebras, arXiv:1406.4191.

\bibitem{JLa} C. Jiang and C.H. Lam, Level-Rank Duality for Vertex Operator Algebras of types $B$ and $D$, arXiv: 1703.04889.

\bibitem{Kac}
V. G. Kac, \emph{Infinite-dimensional Lie Algebras}, 3rd ed., Cambridge
University Press, Cambridge, 1990.

\bibitem{Lam} C.H. Lam, A level-rank duality for parafermion vertex operator algebras of type $A$, {\em Proc. Amer. Math. Soc.}
 {\bf 142} (2014), 4133-4142.

\bibitem{L2} H. Li, The physics superselection principle in vertex operator algebra theory. {\em J. Algebra} \textbf{196}(2) (1997), 436-457.

\bibitem{LL}
J. Lepowsky and H. Li, \emph{Introduction to Vertex Operator Algebras
and Their Representations}, Progress in Math., Vol. 227, Birkh\"auser,
Boston, 2004.

\bibitem{M1} M. Miyamoto, Flatness of Tensor Products and Semi-Rigidity for $C_2$-cofiniteness Vertex Operator Alagebras II, arXiv:0909.3665.

\bibitem{M2}
M. Miyamoto,  $C_2$-Cofiniteness of Cyclic-Orbifold Models, \emph{Comm. Math. Phys.} \textbf{335} (2015),
1279-1286.

\bibitem{Wang} Q. Wang, Automorphism group of parafermion vertex operator algebras, {\em J. Pure Appl. Alg.}
 {\bf 220} (2016), 94-107.





\end{thebibliography}
\end{document}